\documentclass[a4paper,12pt]{article}
\usepackage{comment}
\usepackage{cite}
\usepackage{amsmath}
\usepackage{amssymb}
\usepackage{amsfonts}
\usepackage[T1]{fontenc}
\usepackage[utf8]{inputenc}
\usepackage{graphicx}
\usepackage{fancyhdr}
\usepackage{float}
\usepackage{xcolor}
\usepackage{authblk}
\usepackage{mathrsfs}
\usepackage{empheq}
\usepackage[hyphens]{url}
\usepackage{hyperref}
\usepackage[]{amsthm}
\usepackage{tikz}
\usetikzlibrary{decorations.markings}

\usepackage{geometry}
\theoremstyle{plain}
\newtheorem{theorem}{Theorem}[section]
\newtheorem{proposition}[theorem]{Proposition}

\theoremstyle{definition}
\newtheorem{definition}[theorem]{Definition}

\theoremstyle{remark}

\newtheorem{example}{Example}[section] 

\begin{document}

\title{A combinatorial proof of Jacobi’s elliptic identity via alternating permutations}

\author[$\dagger$]{Jean-Christophe {\sc Pain}\footnote{jean-christophe.pain@cea.fr}\\
\small
$^1$CEA, DAM, DIF, F-91297 Arpajon, France\\
$^2$Universit\'e Paris-Saclay, CEA, Laboratoire Mati\`ere en Conditions Extr\^emes,\\ 
F-91680 Bruy\`eres-le-Ch\^atel, France
}

\date{}

\maketitle

\begin{abstract}
We provide a unified combinatorial framework connecting Entringer numbers, Dumont–Viennot snakes, and elliptically weighted continued fractions, which gives a structural interpretation of the Jacobi elliptic identity
\begin{equation}\label{eq:jacobi}
\mathrm{sn}'(u)=\mathrm{cn}(u)\,\mathrm{dn}(u),
\end{equation}
where $\mathrm{sn}$, $\mathrm{cn}$ and $\mathrm{dn}$ are the Jacobi elliptic functions. This framework allows the decomposition of weighted snakes corresponding to the derivative of $\mathrm{sn}$ into canonical $\mathrm{cn}$- and $\mathrm{dn}$-components, bridging classical combinatorics and elliptic function theory.
\end{abstract}

\section{Introduction}\label{sec1}

Alternating permutations \cite{Stanley2010}, also known as up-down permutations, are classical combinatorial objects that have been studied since the 19th century. In his seminal work, Andr\'e \cite{Andre1881} introduced them and provided initial enumerative results, which were later interpreted combinatorially by Entringer \cite{Entringer1966} through the now well-known Entringer triangle. These numbers enumerate permutations according to the position of their first entry and provide a refinement of Euler numbers. Dumont and Viennot \cite{DumontViennot1979,DumontViennot1980} developed the concept of snakes, a bijective model in which alternating permutations are represented as increasing trees with a ``zigzag'' structure. This representation encodes the recursive growth of alternating permutations and allows for natural combinatorial interpretations of their statistics. Analytic methods reveal that generating functions of alternating permutations naturally lead to continued fraction representations, providing an alternative viewpoint to the combinatorial approaches. 

Continued fractions provide another perspective on Entringer numbers. Flajolet \cite{Flajolet1980} introduced combinatorial interpretations of J-fractions (continued fractions of Jacobi type), while Finch \cite{Finch2003} developed explicit continued fractions for Entringer numbers. The modern framework of analytic combinatorics \cite{FlajoletSedgewick2009} connects these continued fractions to generating functions and recurrences in a systematic way. Elliptic extensions of classical continued fractions appear in Ismail's book ``Classical and quantum orthogonal polynomials in one variable'' \cite{Ismail2005}, providing links between combinatorial structures and elliptic functions. Extensions to $q$-deformed analogues of classical sequences have been investigated by Ismail in Ref. \cite{Ismail2005} and more recently by Josuat-Vergès \cite{Josuat2010}, Gelineau \cite{Gelineau2011}, Jovović \cite{Jovovic2012} and Bernardi and Sportiello \cite{BernardiSportiello2020}, to quote only a few.

Jacobi's elliptic functions $\mathrm{sn}$, $\mathrm{cn}$, and $\mathrm{dn}$ \cite{Jacobi1829,AbramowitzStegun1965} 
are fundamental analytic objects with deep connections to elliptic integrals and complex analysis. Equivalently, Jacobi's elliptic functions can be defined in terms of the theta functions \cite{Whittaker1927}. With $z,\tau \in \mathbb{C}$ such that $\operatorname {Im} \tau >0$, let us set
$$
\theta _{1}(z|\tau )=\displaystyle \sum _{n=-\infty }^{\infty }(-1)^{n-{\frac {1}{2}}}e^{(2n+1)iz+\pi i\tau \left(n+{\frac {1}{2}}\right)^{2}},
$$
$$
\theta _{2}(z|\tau )=\displaystyle \sum _{n=-\infty }^{\infty }e^{(2n+1)iz+\pi i\tau \left(n+{\frac {1}{2}}\right)^{2}},
$$
$$
\theta _{3}(z|\tau )=\displaystyle \sum _{n=-\infty }^{\infty }e^{2niz+\pi i\tau n^{2}},
$$
$$
\theta _{4}(z|\tau )=\displaystyle \sum _{n=-\infty }^{\infty }(-1)^{n}e^{2niz+\pi i\tau n^{2}}
$$
and
$$
\theta _{2}(\tau )=\theta _{2}(0|\tau ), 
$$
$$
\theta _{3}(\tau )=\theta _{3}(0|\tau ), 
$$
$$
\theta _{4}(\tau )=\theta _{4}(0|\tau ). 
$$
Then, with 
$$
K=K(m)=K(k^{2})=\int _{0}^{1}{\frac {dt}{\sqrt {(1-t^{2})(1-mt^{2})}}}=\int _{0}^{1}{\frac {dt}{\sqrt {(1-t^{2})(1-k^{2}t^{2})}}},
$$
as well as $K'=K(1-m)$, $\zeta =\pi u/(2K)$ and $\tau =iK'/K$, we have
\begin{align}
\operatorname {sn} (u,m)&={\frac {\theta _{3}(\tau )\theta _{1}(\zeta |\tau )}{\theta _{2}(\tau )\theta _{4}(\zeta |\tau )}},\\\operatorname {cn} (u,m)&={\frac {\theta _{4}(\tau )\theta _{2}(\zeta |\tau )}{\theta _{2}(\tau )\theta _{4}(\zeta |\tau )}},\\\operatorname {dn} (u,m)&={\frac {\theta _{4}(\tau )\theta _{3}(\zeta |\tau )}{\theta _{3}(\tau )\theta _{4}(\zeta |\tau )}}.
\end{align}

In particular, the differential identity
\[
\mathrm{sn}'(u)=\mathrm{cn}(u)\,\mathrm{dn}(u)
\]
is central to the theory. While the analytic properties of these functions are classical, combinatorial interpretations, especially those connecting elliptic weights with alternating permutations or snakes, are still rare and scattered in the literature. To the best of our knowledge, an explicit combinatorial interpretation of Jacobi's differential identity in terms of elliptically weighted alternating permutations has not appeared in the literature.

In this work, we provide such a unified framework. We show that the derivative of $\mathrm{sn}$ can be interpreted combinatorially as a canonical splitting of elliptically weighted snakes into $\mathrm{cn}$-type and $\mathrm{dn}$-type components. This approach not only reproduces Jacobi's differential identity, but also provides a combinatorial understanding via Entringer numbers, Dumont--Viennot snakes, and elliptic continued fractions. Our framework thus bridges classical combinatorics and the analytic theory of elliptic functions, offering both conceptual and structural insight.

The paper is organized as follows. Section~\ref{sec2} recalls alternating permutations and Entringer numbers, together with their fundamental recurrence and generating function interpretation in terms of Jacobi's elliptic sine. Section~\ref{sec3} introduces an elliptic weighting on alternating permutations and interprets this deformation combinatorially. In the same section, differentiation is reformulated as a growth operator and the two canonical removal mechanisms underlying the structure of weighted permutations are identified. Section~\ref{sec4} establishes the canonical factorization that leads directly to a combinatorial proof of Jacobi's identity. Section~\ref{sec5} embeds the construction into the Dumont--Viennot snake model, and section~\ref{sec6} relates the framework to continued fractions and Entringer refinements, highlighting the compatibility between recursive growth, J-fractions, and elliptic deformation. We present the corresponding elliptic continued fraction expansions for \(\mathrm{sn}\), \(\mathrm{cn}\), and \(\mathrm{dn}\). These expansions reveal their combinatorial meaning in terms of weighted paths and branching processes.

The classical derivation of the exponential generating function of the number of ascending alternating permutations of $\mathfrak{S}_n$ is recalled in Appendix~\ref{appA}, together with their recursion relation and integral representation. Appendix~\ref{appB} develops the snake model in detail, including elliptic weighting and the structural interpretation of Jacobi's identity in terms of snake splitting. 

\paragraph{Notation.}
In the combinatorial sections (Sections~\ref{sec2}--\ref{sec5}), we denote by
\[
\mathrm{sn}_k(u) = \sum_{n\ge0} \left( \sum_{\pi \in \mathcal{S}_{2n+1}} k^{\nu(\pi)} \right) \frac{u^{2n+1}}{(2n+1)!}
\]
the exponential generating function of \emph{elliptically weighted alternating permutations}, where $k$ is the weight of a peak (the combinatorial elliptic parameter), \(\mathcal{S}_{2n+1}\) the set of alternating permutations of size \(2n+1\) and \(\nu(\pi)\) as the number of peaks of \(\pi\). Here, $k$ is a combinatorial weight, not the analytic modulus.

In the analytic section (sec.~\ref{sec6}), we use the classical Jacobi elliptic functions
\[
\mathrm{sn}(u,k),\qquad \mathrm{cn}(u,k),\qquad \mathrm{dn}(u,k),
\]
where $k$ is the usual elliptic modulus. This notation avoids confusion between the combinatorial weight $k$ and the analytic elliptic modulus $k$.

\section{Alternating permutations and Entringer numbers}\label{sec2}

A permutation $\pi\in\mathfrak{S}_n$ is said to be \emph{alternating} if it satisfies
\[
\pi_1<\pi_2>\pi_3<\pi_4>\cdots.
\]
Alternating permutations are enumerated by the Euler numbers. A refined enumeration is given by the \emph{Entringer numbers} $E(n,k)$, defined as the number of alternating permutations of size $n$ whose first value is equal to $k+1$. The Entringer numbers form a triangular array satisfying the recurrence relation
\begin{equation}\label{eq:entringer-recurrence}
E(n,k)=E(n,k-1)+E(n-1,n-k),
\end{equation}
with boundary conditions $E(0,0)=1$ and $E(n,0)=0$ for $n>0$. This recurrence admits a natural interpretation in terms of the insertion of the maximal element into an alternating permutation of smaller size. Since the Entringer numbers refine alternating permutations according to the first value, summing over $k$ gives the total number of alternating permutations:
\[
\sum_{k=0}^{2n} E(2n+1,k)=T_{2n+1}.
\]
It is classical (Andr\'e's theorem, see Appendix \ref{appA}) that the exponential generating function of alternating permutations of odd size is the sine function:
\[
\sum_{n\ge0} T_{2n+1}\frac{u^{2n+1}}{(2n+1)!}=\sin u.
\]
We denote by $E(n,k)$ the classical Entringer numbers. Later, in Section~\ref{sec3}, these numbers will be extended by introducing an elliptic weight to form $E_k(.,.)$. The corresponding exponential generating function of elliptically weighted Entringer numbers becomes
\[
\mathrm{sn}_k(u)
=
\sum_{n\ge0}
\left(\sum_{j=0}^{2n} E_k(2n+1,j)\right)
\frac{u^{2n+1}}{(2n+1)!}.
\]
Introducing the elliptic peak weight replaces $\sin$ by Jacobi's $\mathrm{sn}$, yielding the weighted exponential generating function above. When $k=0$, one recovers the classical sine generating function.

\begin{table}[!ht]
\centering
\renewcommand{\arraystretch}{1.2}
\begin{tabular}{|c|c|c|c|c|c|c|}
\hline
$n \setminus k$ & \textbf{0} & \textbf{1} & \textbf{2} & \textbf{3} & \textbf{4} & \textbf{5} \\ \hline
\textbf{0} & 1 &   &   &   &   &   \\ \hline
\textbf{1} & 0 & 1 &   &   &   &   \\ \hline
\textbf{2} & 0 & 1 & 1 &   &   &   \\ \hline
\textbf{3} & 0 & 1 & 2 & 2 &   &   \\ \hline
\textbf{4} & 0 & 2 & 4 & 5 & 5 &   \\ \hline
\textbf{5} & 0 & 5 & 10 & 14 & 16 & 16 \\ \hline
\end{tabular}
\caption{Entringer numbers $E(n,k)$ obtained from recurrence relation (\ref{eq:entringer-recurrence}) for small values of $n$ and $k$.}\label{tabent}
\end{table}

Table \ref{tabent} represents the triangle of Entringer numbers for small values of $n$.

\section{Elliptic weighting, local statistics and derivation as a growth operator}\label{sec3}

We define $\mathcal C$ (resp.\ $\mathcal D$) as the class of even-sized alternating permutations whose last (resp.\ first) comparison is a descent, i.e.
\[
\mathcal C_{2n}=\{\pi\in\mathfrak S_{2n}:\pi_1<\pi_2>\cdots>\pi_{2n}\},
\qquad
\mathcal D_{2n}=\{\pi\in\mathfrak S_{2n}:\pi_1>\pi_2<\cdots<\pi_{2n}\}.
\]
Their weighted exponential generating functions are denoted \(\mathrm{cn}_k(u)\) and \(\mathrm{dn}_k(u)\), defined as
\[
\mathrm{cn}_k(u)=\sum_{n\ge0}
\left(\sum_{\pi\in\mathcal C_{2n}} k^{\nu(\pi)}\right)
\frac{u^{2n}}{(2n)!},
\qquad
\mathrm{dn}_k(u)=\sum_{n\ge0}
\left(\sum_{\pi\in\mathcal D_{2n}} k^{\nu(\pi)}\right)
\frac{u^{2n}}{(2n)!}.
\]
Let $\mathcal{C}_{2j-2}$ denote the class of alternating permutations obtained from the left subsequence $L$ after removing the maximal element, and $\mathcal{D}_{2(n-j)}$ denote the class obtained from the right subsequence $R$.

Let \(\mathcal{S}_{2n+1}\) denote the set of alternating permutations of size \(2n+1\). Marking the maximal element in an element \(\pi \in \mathcal{S}_{2n+1}\) and removing it corresponds to splitting the permutation into a pair \((\alpha, \beta)\) with \(\alpha \in \mathcal{C}_{2j-2}\) and \(\beta \in \mathcal{D}_{2(n-j)}\).

We define \(\nu(\pi)\) as the number of peaks of \(\pi\). The \emph{elliptic weight} of \(\pi\) is then
\[
w(\pi) = k^{\nu(\pi)},
\]
which is multiplicative under the decomposition \(\pi \mapsto (\alpha, \beta)\) since the number of peaks in \(\pi\) is the sum of the peaks in \(\alpha\) and \(\beta\) plus the marked maximal peak.

Summing over all such marked permutations gives the exponential generating function identity
\[
\mathrm{sn}_k'(u) = \sum_{n\ge0}\sum_{\sigma\in\mathcal{S}_{2n+1}^\bullet} w(\sigma)\frac{u^{2n}}{(2n)!},
\]
where $\mathcal{S}_{2n+1}^\bullet$ denotes the class of weighted alternating permutations with a marked maximal element. This prepares the canonical factorization that will be formalized in Section~\ref{sec4}.

\section{Canonical factorization and proof of Jacobi's identity}\label{sec4}

\begin{proposition}[Canonical factorization of elliptic snakes]\label{prop:factorization}

By the locality assumption on the elliptic statistic (section~\ref{sec3}), the weight is multiplicative under removal of the maximal element.

Let $\mathcal{S}_{2n+1}$ denote the set of alternating permutations
\[
\sigma=\sigma_1\sigma_2\cdots\sigma_{2n+1}\in\mathfrak{S}_{2n+1}
\]
satisfying
\[
\sigma_1<\sigma_2>\sigma_3<\cdots>\sigma_{2n+1}.
\]
Let $\mathcal{S}_{2n+1}^\bullet$ be the corresponding class where one marks the maximal element. Then removal of the maximal element induces a weight-preserving bijection
\[
\bigsqcup_{n\ge0}\mathcal{S}_{2n+1}^\bullet
\;\simeq\;
\left(\bigsqcup_{j\ge0}\mathcal{C}_{2j}\right)
\times
\left(\bigsqcup_{m\ge0}\mathcal{D}_{2m}\right),
\]
where $\mathcal{C}$ and $\mathcal{D}$ are respectively the combinatorial classes enumerated by $\mathrm{cn}_k(u)$ and $\mathrm{dn}_k(u)$. Consequently,
\[
\mathrm{sn}_k'(u)=\mathrm{cn}_k(u)\,\mathrm{dn}_k(u).
\]
\end{proposition}

\begin{proof}

Let $\sigma\in\mathcal{S}_{2n+1}$ and let $M=2n+1$ be its maximal element. Since the permutation is alternating and starts with an ascent, peaks occur exactly at even positions. Because $M$ is larger than its neighbors, it must occur at some even position $2j$ with $1\le j\le n$. Thus $M$ is always located at a peak.

Let us now write
\[
\sigma = L\, M\, R,
\]
where $L$ (resp. $R$) is the subsequence to the left (resp. right) of $M$. Removing $M$ produces two words $L$ and $R$ of respective lengths $2j-1$ and $2(n-j)$. Deleting the last element of $L$ (which is forced by alternation) and standardizing yields an element
\[
\alpha\in\mathcal C_{2j-2}.
\]
Similarly, standardizing $R$ yields
\[
\beta\in\mathcal D_{2(n-j)}.
\]
By inspection of the alternating pattern, the left block has odd length and ends with a descent. After standardization, this produces an even-sized alternating permutation whose last comparison is a descent, i.e.\ an element of $\mathcal C$. Similarly, the right block begins with a descent and standardizes to an element of $\mathcal D$.

Since the maximal element is removed at a peak, no alternation constraint is broken inside $L$ or $R$. Hence both $\alpha$ and $\beta$ are alternating permutations of the appropriate type.

Conversely, given a pair $(\alpha,\beta)$ with
\[
\alpha\in\mathcal{C}_{2j-2},\qquad\beta\in\mathcal{D}_{2(n-j)},
\]
one reconstructs $\sigma$ uniquely by inserting a new maximal element between the relabelled copies of $\alpha$ and $\beta$, after shifting the labels of $\beta$ appropriately. The alternation pattern forces a unique insertion point. Thus the mapping is bijective.

Let \(\nu(\sigma)\) denote the number of peaks of \(\sigma\), and define the elliptic weight
\[
w(\sigma) = k^{\nu(\sigma)}.
\]
Because \(\nu\) is local — that is, it depends only on the positions of the peaks — removing the maximal element separates \(\sigma\) into two independent alternating blocks \(\alpha\) and \(\beta\). Hence the weight is multiplicative:
\[
w(\sigma) = w(\alpha)\, w(\beta)\, k,
\]
where the extra factor \(k\) accounts for the marked maximal element, which is always a peak. 
The additional factor \(k\) contributed by the removed maximal peak is exactly accounted for by the marking operation defining the derivative of the exponential generating function. Hence the bijection is weight-preserving. Summing over all sizes gives

\[
\sum_{n\ge0}\sum_{\sigma\in\mathcal{S}_{2n+1}^\bullet}w(\sigma)\frac{u^{2n}}{(2n)!}
=
\left(
\sum_{j\ge0}\sum_{\alpha\in\mathcal{C}_{2j}}w(\alpha)\frac{u^{2j}}{(2j)!}
\right)
\left(
\sum_{m\ge0}\sum_{\beta\in\mathcal{D}_{2m}}w(\beta)\frac{u^{2m}}{(2m)!}
\right).
\]
By definition of the exponential generating functions, this identity reads
\[
\mathrm{sn}_k'(u)=\mathrm{cn}_k(u)\,\mathrm{dn}_k(u).
\]
\end{proof}

\begin{example}
For instance,
\[
\sigma=2\,5\,1\,7\,3\,6\,4
\]
has maximal element $7$ at position $4$. Removing it gives
\[
L=2\,5\,1,
\qquad
R=3\,6\,4,
\]
which standardize to alternating permutations of respective types.
\end{example}

\section{Connection with Dumont--Viennot snakes}\label{sec5}

The combinatorial interpretation presented above fits naturally into the framework of \emph{snakes}, introduced by Dumont and Viennot. A snake is a labeled alternating tree encoding an alternating permutation together with its recursive growth structure. More precisely, Dumont and Viennot showed that alternating permutations are in bijection with increasing trees whose nodes carry an alternating sign constraint. Under this correspondence, local extrema of the permutation correspond to branching points of the tree, while the insertion of the maximal element corresponds to the addition of a new leaf. In this model, the elliptic weight introduced in Section~\ref{sec3} has a natural interpretation: each weighted peak corresponds to a distinguished branching in the snake, whose contribution is marked by a factor $k$. Valley-type insertions correspond to unweighted branches. The decomposition of proposition~\ref{prop:factorization} admits a transparent interpretation in the snake model: removing the maximal label separates the snake into two independent sub-snakes, one encoding pure alternation (associated with $\mathrm{cn}$), and the other encoding elliptic branching (associated with $\mathrm{dn}$). This explains combinatorially why the derivative of $\mathrm{sn}$ factorizes as the product $\mathrm{cn}\times\mathrm{dn}$.

\section{Jacobi continued fractions and Entringer refinements}\label{sec6}

At this stage we pass from the combinatorial generating functions $\mathrm{sn}_k,\mathrm{cn}_k,\mathrm{dn}_k$
to the analytic Jacobi functions $\mathrm{sn}(u,k),\mathrm{cn}(u,k),\mathrm{dn}(u,k)$, which satisfy the same differential system.

An alternative and complementary viewpoint is provided by continued fractions. The generating functions of Entringer numbers admit well-known continued fraction expansions of Jacobi type. In particular, the ordinary generating function
\[
\sum_{n\ge0} E(2n+1)\,z^{2n+1}
\]
admits a continued fraction expansion whose coefficients encode the same insertion mechanism as the recurrence \eqref{eq:entringer-recurrence}. When elliptic weights are introduced, these continued fractions deform naturally into \emph{elliptic continued fractions}, whose coefficients alternate between weighted and unweighted terms. This alternation reflects precisely the two types of admissible growth steps identified in Section~\ref{sec3}. From this perspective, Jacobi's identity expresses the factorization of the underlying continued fraction at the level of first convergents: differentiation removes the initial coefficient, revealing a product of two continued fractions corresponding respectively to $\mathrm{cn}$ and $\mathrm{dn}$. This interpretation makes explicit the deep compatibility between the snake model, the Entringer triangle, and the analytic structure of Jacobi's elliptic functions. The snake interpretation and the continued fraction expansions provide two complementary combinatorial realizations of Jacobi's elliptic identities. While the snake model emphasizes bijections and growth processes, continued fractions encode the same phenomena algebraically through recursive decompositions. Together, they form a unified combinatorial framework underlying Jacobi's theory.

The canonical decomposition established in proposition~\ref{prop:factorization} places elliptic snakes within the general framework of combinatorial classes admitting recursive insertion of a distinguished extremal element. It is a classical result of Flajolet~\cite{Flajolet1980} that such recursive structures lead naturally to Jacobi continued fractions for their generating functions.

\subsection{Recursive structure and differential relations}\label{subsec61}

Let $S_n$ denote the total elliptic weight of $\mathcal{S}_{2n+1}$, the set of elliptic snakes of size $2n+1$. Removing the maximal element yields the convolution
\[
S_n = \sum_{j=0}^{n} C_j D_{\,n-j},
\]
where $C_j$ and $D_j$ are weighted counts corresponding to the classes enumerated by the analytic Jacobi functions $\mathrm{cn}(u,k)$ and $\mathrm{dn}(u,k)$, respectively.

Passing to exponential generating functions gives the differential identity
\[
\mathrm{sn}'(u,k) = \mathrm{cn}(u,k)\,\mathrm{dn}(u,k),
\]
together with
\[
\mathrm{cn}'(u,k)=-\mathrm{sn}(u,k)\mathrm{dn}(u,k), \qquad
\mathrm{dn}'(u,k)=-k^2 \mathrm{sn}(u,k)\mathrm{cn}(u,k),
\]
which uniquely determine $\mathrm{sn}(u,k)$. The convolution reflects the two types of admissible growth steps identified in Section~\ref{sec3}: adding a new maximal element either extends a horizontal/valley step (contributing to $\mathrm{cn}$) or inserts a peak (elliptically weighted, contributing to $\mathrm{dn}$).

\subsection{Jacobi continued fraction for $\mathrm{sn}(u,k)$}\label{subsec62}

In the elliptic deformation, peak-type insertions acquire an additional weight depending on $k$, leading to quadratic numerator coefficients of the form
\(
\beta_n \propto n^2 k^2
\) 
in the associated J-fraction, while non-branching horizontal/valley steps contribute
\[
\alpha_n = 2n-1.
\]
Consequently, the exponential generating function of elliptic snakes admits the
J-fraction
\[
\mathrm{sn}(u,k)
=
\cfrac{u}{
1-\cfrac{\beta_1 u^2}{
\alpha_1-\cfrac{\beta_2 u^2}{
\alpha_2-\cfrac{\beta_3 u^2}{
\alpha_3-\ddots}}}}.
\]
Substituting the explicit coefficients gives
\[
\mathrm{sn}(u,k)
=
\cfrac{u}{
1-\cfrac{1^2 k^2\,u^2}{
3-\cfrac{2^2 k^2\,u^2}{
5-\cfrac{3^2 k^2\,u^2}{
7-\cdots}}}}
\]
Setting $k=0$ recovers the classical sine expansion:
\[
\sin u = \cfrac{u}{1 - \cfrac{1^2 u^2}{3 - \cfrac{2^2 u^2}{5 - \cfrac{3^2 u^2}{7 - \ddots}}}},
\]
corresponding to Euler–Entringer enumeration of alternating permutations \cite{Foata1970}.

For $k=1$, one obtains the hyperbolic tangent expansion:
\[
\tanh u = \cfrac{u}{1 - \cfrac{2 u^2}{3 - \cfrac{4 u^2}{5 - \cfrac{6 u^2}{7 - \ddots}}}}.
\]

\subsection{Analytic Jacobi fractions and comparison with the combinatorial model}\label{subsec63}

The Jacobian elliptic function $\mathrm{sn}(u,k)$ also admits a classical analytic J-fraction expansion (Jacobi–Stieltjes type). Its coefficients differ from the combinatorial elliptic fraction of Section~\ref{subsec62}. The analytic and combinatorial fractions share the same quadratic growth 
$n^2k^2$ in the numerators, but arise from different structural interpretations: differential equations versus peak insertions. This fraction has a natural combinatorial interpretation in terms of weighted paths: each numerator $\beta_n$ corresponds to a peak step with weight $n^2 k^2$, while each denominator $\alpha_n$ corresponds to a horizontal or valley step of weight $2n-1$.

\vspace{2mm}
For the functions $\mathrm{cn}(u,k)$ and $\mathrm{dn}(u,k)$, there exist analogous J-fraction expansions. These fractions encode similar combinatorial structures — horizontal and peak steps — but their numerator and denominator sequences are not simply equal to the $\alpha_n$ and $\beta_n$ of $\mathrm{sn}$. In particular:
\[
\mathrm{cn}(u,k) = 
\cfrac{1}{1 - \cfrac{\gamma_1 u^2}{\delta_1 - \cfrac{\gamma_2 u^2}{\delta_2 - \cdots}}}, 
\qquad
\mathrm{dn}(u,k) = 
\cfrac{1}{1 - \cfrac{\eta_1 u^2}{\theta_1 - \cfrac{\eta_2 u^2}{\theta_2 - \cdots}}} \,,
\]
where $\gamma_n, \delta_n, \eta_n, \theta_n$ are sequences determined by the differential equations of $\mathrm{cn}$ and $\mathrm{dn}$, but their coefficients are known and involve quadratic polynomials in $n$ depending on $k^2$.

Combinatorially, one can still interpret these fractions in terms of weighted paths and branching processes, with horizontal steps corresponding to the denominators and peak steps to the numerators. Thus, the continued fractions arising from Jacobi elliptic function theory and those derived from elliptically weighted alternating permutations are structurally analogous but not identical: they share the same recursive skeleton (Jacobi fractions) while encoding different weights. The combinatorial fraction reflects peak insertions in snakes, whereas the analytic fraction reflects the elliptic modulus in Jacobi’s differential system. The bridge between the two frameworks is provided by the differential identity
\[
\frac{d}{du}\mathrm{sn}(u,k)=\mathrm{cn}(u,k)\mathrm{dn}(u,k),
\]
which appears combinatorially as the canonical factorization of elliptic snakes.

\subsection{Combinatorial interpretation of coefficients}\label{subsec64}

The \(\alpha_n\) coefficients correspond to horizontal steps in the snake (valley-type), unweighted, contributing to $\mathrm{cn}(u)$, while the \(\beta_n\) coefficients represent branching steps (peak-type), elliptically weighted, contributing to $\mathrm{dn}(u)$. This encoding makes explicit the two canonical growth mechanisms of elliptic snakes. Introducing the elliptic weight $k$ modifies only the $\beta_n$ coefficients, producing an elliptically weighted continued fraction.

\subsection{Factorization under differentiation}\label{subsec65}

Differentiation of the J-fraction removes the initial coefficient and yields
\[
\mathrm{sn}'(u) = \mathrm{cn}(u)\,\mathrm{dn}(u),
\]
revealing the multiplicative factorization at the level of first convergents. From the combinatorial viewpoint, this corresponds to decomposing an alternating permutation at its maximal peak into two independent substructures encoded by $\mathrm{cn}(u)$ and $\mathrm{dn}(u)$.

\section{Conclusion}\label{sec7}

The study of alternating permutations and their enumeration goes back to André \cite{Andre1881}, with combinatorial interpretations formalized by Entringer \cite{Entringer1966}. Jacobi's differential identity \eqref{eq:jacobi} admits a direct combinatorial interpretation: it expresses a fundamental factorization of the growth process of alternating permutations of odd size into two independent mechanisms. Within this framework, Jacobi’s elliptic functions arise naturally as exponential generating functions of elliptically weighted alternating permutations. The snake interpretation and the J-fraction expansions provide two complementary realizations of Jacobi's elliptic identities: the first emphasizes bijections and growth processes, the second encodes the same phenomena algebraically through recursive decomposition.

\appendix

\section{The number of ascending alternating permutations and Andr\'e's theorem}\label{appA}

Let $A_n$ be the number of ascending alternating permutations of the symmetric group $\mathfrak{S}_n$, and $A_0 = A_1 = 1$. The number $A_n$ is related to the Entringer ones by $A_n=E(n,n)$ , while 
\[
T_{2n+1} = \sum_{k=0}^{2n} E(2n+1,k).
\]
gives all alternating permutations of size. We indeed have $A_2 = 1$ (there is only one ascending permutation for $n=2$: the identity). To construct an alternating permutation $\sigma$ of $\{1, \dots, n+1\}$, we can choose the integer $\sigma^{-1}(n+1) = k+1 \in [1, n+1]$. The structure of the permutation is as follows:
\[
\sigma(1) < \sigma(2) > \dots > \sigma(k) < \sigma(k+1) = n+1 > \sigma(k+2) < \dots < \sigma(n+1).
\]
We choose $k$ elements, forming a set $E$, from $[1, n]$, which will be in $\sigma([1, k])$. There are $\binom{n}{k}$ possibilities for this choice. Using these $k$ elements, we construct an alternating bijection from $[1, k]$ to $E$ satisfying $\sigma(k-1) > \sigma(k)$. The number of such bijections is exactly $A_k$. Then, we complete the construction with an alternating bijection from $[k+2, n+1]$ to $[1, n] \setminus E$ satisfying $(\sigma(k+1)=n+1) > \sigma(k+2) < \sigma(k+3)$. Thus, we obtain the following recurrence relation:
\begin{equation*}
2A_{n+1} = \sum_{k=0}^{n} \binom{n}{k} A_k A_{n-k}.
\end{equation*}
We have 
$$
A_n \leq \frac{|\sigma_n|}{2} = \frac{n!}{2} < n!.
$$
According to the comparison theorem for series with positive terms, the series $\sum \frac{A_n x^n}{n!}$ converges for $|x| < 1$, so its radius of convergence is $R \geq 1$. Let 
$$
f(x) = \sum_{n=0}^{+\infty} \frac{A_n x^n}{n!}.
$$
By performing a Cauchy product on $]-R, R[$, we obtain:
\[ 
f^2(x) = \sum_{n=0}^{+\infty} \frac{x^n}{n!} \left( \sum_{k=0}^{n} \binom{n}{k} A_k A_{n-k} \right). 
\]
Using the recurrence relation:
\[ 
f^2(x) = 1 + 2 \sum_{n=1}^{+\infty} \frac{A_{n+1} x^n}{n!} = 1 + 2(f'(x) - 1) = 2f'(x) - 1,
\]
we obtain the differential equation 
$$
\frac{f'(x)}{1+f^2(x)} = \frac{1}{2}.
$$ 
By integration, we get
\[
\arctan(f(x)) = \frac{x}{2} + \frac{\pi}{4} \quad (\text{since } f(0)=1),
\]
yielding, for $x \in ]-\frac{3\pi}{2}, \frac{\pi}{2}[$, the well-known result (sometimes referred to as ``the Andr\'e theorem''):
\[ 
f(x) = \tan\left(\frac{x}{2} + \frac{\pi}{4}\right) = \frac{1 + \tan(x/2)}{1 - \tan(x/2)} = \tan x + \frac{1}{\cos x} .
\]
These numbers $A_n$ of ascending alternating permutations of the symmetric group $\mathfrak{S}_n$ are called ``secant-tangent numbers'' because they appear in the power-series expansions \cite{Carlitz1960}:
$$
\tan x = \sum_{n=0}^{+\infty} A_{2n+1} \frac{x^{2n+1}}{(2n+1)!}
$$
and
$$
\sec x = \frac{1}{\cos x} = \sum_{n=0}^{+\infty} A_{2n} \frac{x^{2n}}{(2n)!}.
$$
The $A_{2n}$ are the Euler numbers $E_{2n}$ and the $A_{2n+1}$ are related to the Bernoulli numbers through
\[ 
A_{2n+1} = (-1)^n \frac{2^{2n+2}(2^{2n+2}-1)}{2n+2} B_{2n+2}.
\]
It can be shown that $f^{(n)}(x) \geq 0$ for all $x \in [0, \frac{\pi}{2}[$. Using Taylor's formula with integral remainder, the series converges on this interval, hence $R = \frac{\pi}{2}$. According to D'Alembert's rule (ratio test) for power series:
\[ 
\frac{\pi}{2} = \lim_{n \to +\infty} \frac{(n+1)A_n}{A_{n+1}}.
\]
For example, for $n=6$, $A_5=16$ and $A_6=61$, we have $\frac{6 \times 16}{61} \approx 1.573$.

It is worth mentioning that an explicit formula for the $A_n$ numbers can be easily obtained. For that purpose, let us set
\[
g(x) = f(x) - 1.
\]
Then
\[
g'(x) = \frac{1}{2} \big((g(x)+1)^2 + 1\big) = \frac{1}{2}\big(g(x)^2 + 2 g(x) + 2\big).
\]
The latter identity can be recast into
\[
g'(x) = g(x) + \frac{1}{2} g(x)^2 + 1.
\]
After simplification (absorbing constants into a shift), we can write it in the standard form for Lagrange inversion:
\[
g(x) = x \, \phi(g(x)),
\]
for some analytic function $\phi$. More explicitly, using the substitution $h(x) = g(x)/2$ leads to
\[
h(x) = \frac{e^x}{2 - e^x} \quad \Longleftrightarrow \quad x = \log\frac{2h(x)}{1+h(x)}.
\]
By the \emph{Lagrange inversion formula}, we have
\[
[x^n] h(x) = \frac{1}{n} [t^{n-1}] \phi(t)^n.
\]
Here, $\phi(t) = \frac{1}{2}(1+t)$, giving
\[
h(x) = \sum_{n=1}^\infty \frac{1}{n} [t^{n-1}] \left(\frac{1+t}{2}\right)^n x^n.
\]
Next, let us expand $(1+t)^n$ using Stirling numbers of the second kind $S(n,k)$; we get
\[
(1+t)^n = \sum_{k=0}^n \binom{n}{k} t^k.
\]
Accounting for the factor $2^{-n}$, we obtain
\[
[x^n] h(x) = \frac{1}{n} \sum_{k=0}^{n-1} \binom{n}{k} 2^{-n} S(n-1,k) = \sum_{k=0}^{n-1} \frac{1}{k+1} \binom{n}{k} 2^{-(k+1)} S(n,k).
\]
Multiplying by $n!$ to recover the ordinary numbers $A_n$, we finally have
\[
A_n = n! \sum_{k=0}^{n} \frac{1}{k+1} \binom{n}{k} 2^{-(k+1)} S(n,k),
\]
which is the sought-after expression. This derivation relies on the differential equation satisfied by $\tan x + \sec x$, the substitution $g(x) = h(x)-1$ to bring it into Lagrange inversion form, the Lagrange inversion formula, and the expansion in terms of Stirling numbers of the second kind.

Furthermore, since we have
$$
S(n,k)=\left\{{n \atop k}\right\}={\frac {1}{k!}}\sum _{i=0}^{k}(-1)^{k-i}{\binom {k}{i}}i^{n}=\sum _{i=0}^{k}{\frac {(-1)^{k-i}i^{n}}{(k-i)!i!}},
$$
we obtain the alternative expression
$$
A_n=n!\sum_{k=0}^{n}\frac{2^ {-(k+1)}}{(k+1)!}\sum_{j=0}^ k(-1)^{k-j}\binom{k}{j}j^n.
$$
An integral representation can also be obtained for the $A_n$ coefficients \cite{Comtet1970}. We start from the exponential generating function
\[
f(x) = \sum_{n=0}^{\infty} A_n \frac{x^n}{n!} = \tan x + \sec x.
\]
By the Cauchy integral formula for exponential generating functions, we have
\[
A_n = n! \, [x^n] f(x) = \frac{n!}{2\pi i} \oint_{|x|=\epsilon} \frac{\tan x + \sec x}{x^{n+1}} \, dx,
\]
where $\epsilon>0$ is sufficiently small and the contour encircles the origin once counterclockwise. Set $x = \arctan t$, so that $dx = dt/(1+t^2)$ and
\[
\tan x = t, \quad \sec x = \sqrt{1+t^2}.
\]
The contour in the $x$-plane maps to a contour around $t=0$, giving
\[
A_n = \frac{n!}{2\pi i} \oint_{|t|=\delta} \frac{t + \sqrt{1+t^2}}{(\arctan t)^{\,n+1}} \frac{dt}{1+t^2}.
\]
This is a valid integral representation, though not fully elementary. Alternatively, one can use Fourier inversion. Since
\[
\sec t + \tan t = \frac{1 + \sin t}{\cos t} = \frac{2 \cos^2(t/2)}{\cos t} \, e^{i \theta} \quad \text{(complex form)},
\]
one can write
\[
A_n = \frac{2 n!}{\pi} \, \text{Im} \int_0^{\pi/2} \frac{e^{-i n t}}{\cos^{\,n+1} t} \, dt,
\]
or, more elegantly, using a real integral (analogous to Catalan numbers):
\[
A_n = \frac{2 n!}{\pi} \int_0^{\infty} \frac{y^n}{\cosh^{\,n+1} y} \, dy.
\]
This last form is particularly convenient for asymptotic analysis, because the integrand is positive, and Laplace's method can be applied to estimate $A_n$ for large $n$. We have the following integral representations for the alternating numbers $A_n$:
\[
A_n = \frac{n!}{2\pi i} \oint \frac{\tan x + \sec x}{x^{n+1}} \, dx
= \frac{2 n!}{\pi} \int_0^{\infty} \frac{y^n}{\cosh^{\,n+1} y} \, dy.
\]
These formulas are analogous to the classical integral representations of Catalan numbers:
\[
C_n = \frac{1}{2\pi i} \oint \frac{1 - \sqrt{1-4x}}{2x^{n+1}} \, dx
= \frac{4^n}{\pi} \int_0^1 x^n \sqrt{\frac{1-x}{x}} \, dx.
\]

\section{Snakes and elliptic weights}\label{appB}

\subsection{Dumont--Viennot snakes}

Dumont and Viennot introduced \emph{snakes} as a combinatorial model encoding alternating permutations together with their recursive structure. A snake is an increasing labeled tree whose nodes are arranged along a zigzag skeleton reflecting the alternation constraint. More precisely, a snake of size $n$ is a rooted increasing tree on $\{1,\dots,n\}$ such that the children of each node are linearly ordered, and the parity of the depth determines whether the ordering is increasing or decreasing. This structure is in bijection with alternating permutations of size $n$. Under this correspondence, leaves correspond to local extrema of the permutation, the root corresponds to the first element, and them insertion of the maximal label corresponds to grafting a new leaf. This bijection preserves natural statistics such as the number and position of local extrema.

\subsection{Snake growth, Entringer recurrence and weighted nodes}

The Entringer recurrence
\[
E(n,k)=E(n,k-1)+E(n-1,n-k)
\]
admits a direct interpretation in the snake model. The two terms correspond to two distinct growth operations: insertion of the new maximal label as a child of an existing node without changing the orientation of the local ordering, and the insertion as a child that reverses the local ordering. These two growth mechanisms correspond exactly to valley-type and peak-type insertions in the permutation model.

\subsection{Elliptic weights on snakes}

To incorporate the elliptic modulus, we introduce a weight on snakes. A node is said to be \emph{elliptic} if it corresponds to a peak-type local extremum in the associated permutation.

\begin{definition}
Let $S$ be a snake. We define its elliptic weight by
\[
w(S)=k^{\nu(S)},
\]
where $\nu(S)$ is the number of elliptic nodes of $S$.
\end{definition}

This weighting is compatible with the growth process: grafting a new leaf at an elliptic node multiplies the weight by $k$, while grafting at a non-elliptic node leaves the weight unchanged.

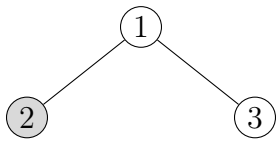
\begin{figure}[!ht]
\centering
\begin{tikzpicture}[level distance=1.2cm,
  every node/.style={circle,draw,inner sep=2pt},
  elliptic/.style={fill=gray!30},
  level 1/.style={sibling distance=3cm},
  level 2/.style={sibling distance=1.5cm}]
\node{1}
  child{node[elliptic]{2}}
  child{node{3}};
\end{tikzpicture}
\caption{A simple Dumont--Viennot snake corresponding to an alternating permutation. The elliptic node (shaded) contributes to a factor $k$ to the weight.}
\label{fig:snake-elliptic}
\end{figure}

Figure~\ref{fig:snake-elliptic} illustrates a minimal snake structure with one elliptic node. The associated weight is $k$. The zigzag ordering is encoded in the planar embedding of the tree.

\subsection{Snakes and Jacobi functions}

For $n\ge0$, let $\mathcal{S}_{2n+1}$ denote the class of weighted snakes of odd size $2n+1$. The exponential generating functions of these classes are given by:
\begin{align*}
\mathrm{sn}(u) &= \sum_{n\ge0} \left(\sum_{S\in\mathcal{S}_{2n+1}} w(S)\right)
\frac{u^{2n+1}}{(2n+1)!},\\
\mathrm{cn}(u) &= \sum_{n\ge0} \left(\sum_{S\in\mathcal{C}_{2n}} 1\right)
\frac{u^{2n}}{(2n)!},\\
\mathrm{dn}(u) &= \sum_{n\ge0} \left(\sum_{S\in\mathcal{D}_{2n}} w(S)\right)
\frac{u^{2n}}{(2n)!},
\end{align*}
where $\mathcal{C}_{2n}$ and $\mathcal{D}_{2n}$ denote the subclasses corresponding respectively to non-elliptic and elliptic growth.

\subsection{Combinatorial interpretation of Jacobi's identity}

Differentiation of $\mathrm{sn}(u)$ corresponds to marking a leaf of a snake and removing the maximal label. This operation splits the snake into two independent sub-snakes, as illustrated in Figure~\ref{fig:snake-split-correct}.

\begin{figure}[!ht]
\centering
\begin{tikzpicture}[level distance=1.2cm,
  every node/.style={circle,draw,inner sep=2pt},
  level 1/.style={sibling distance=3cm},
  level 2/.style={sibling distance=1.5cm}]
\node{1}
  child{node{2}}
  child{node{3}}
  child{node{4}};
\draw[dashed] (-0.4,0.2)--(0.4,0.2);
\end{tikzpicture}
\end{figure}

\begin{figure}[!ht]
\centering
\begin{tikzpicture}[scale=1.1]

\coordinate (A) at (0,0);
\coordinate (B) at (1,1);
\coordinate (C) at (2,0);
\coordinate (D) at (3,1.6); 
\coordinate (E) at (4,0);
\coordinate (F) at (5,1);
\coordinate (G) at (6,0);

\draw[thick] (A)--(B)--(C)--(D)--(E)--(F)--(G);

\filldraw[red] (D) circle (2pt);
\node[above] at (D) {$\max$};

\draw[dashed] (3,-0.3)--(3,1.9);

\node at (1,-0.6) {Left sub-snake};
\node at (5,-0.6) {Right sub-snake};

\end{tikzpicture}
\caption{Removal of the maximal element splits a snake into two independent alternating sub-snakes.}
\label{fig:snake-split-correct}
\end{figure}
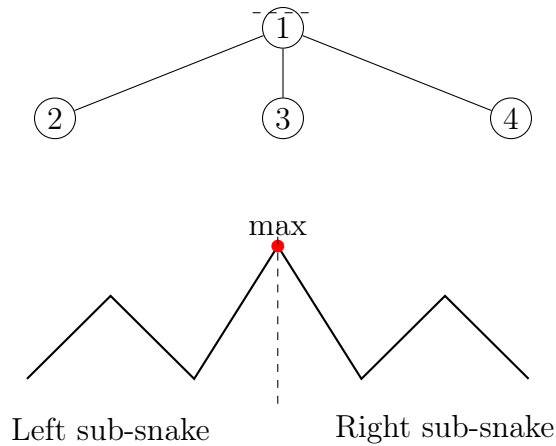

One component records pure alternation and contributes $\mathrm{cn}(u)$, while the other records elliptic branching and contributes $\mathrm{dn}(u)$. This yields the combinatorial factorization
\[
\mathrm{sn}'(u)=\mathrm{cn}(u)\,\mathrm{dn}(u),
\]
thus completing the combinatorial proof of Jacobi's identity. The snake model provides a unified combinatorial framework connecting Entringer numbers, elliptic weights, continued fractions, and Jacobi's elliptic functions. It offers a geometric explanation of why Jacobi's differential identities take a multiplicative form, reflecting the intrinsic decomposability of snake growth.

\end{document}